%% file: hinfinity.tex
\documentclass[journal]{IEEEtran}

\usepackage{amsmath} 
\usepackage{amssymb}  

\usepackage{amsthm}
\usepackage{algpseudocode}
\usepackage{caption}
\usepackage{subcaption}
\usepackage{graphicx}
\usepackage{epstopdf}
\newtheorem{theorem}{Theorem}
\newtheorem{lemma}[theorem]{Lemma}
\newtheorem{proposition}[theorem]{Proposition}
\newtheorem{corollary}[theorem]{Corollary}

\newtheorem{example}{Example}
\newtheorem{remark}{Remark}
\input{./tex/defs}

\begin{document}

\title{$H_{\infty}$ Analysis Revisited}
%
%
%

\author{Seungil~You,~\IEEEmembership{Student member,~IEEE,}
        and~Ather~Gattami,~\IEEEmembership{Member,~IEEE}
\thanks{S. You is with the Control and Dynamical Systems, California Institute of Technology, Pasadena, CA, 91125 USA e-mail: syou@caltech.edu.}
\thanks{A. Gattami is with Ericsson Research, Stockholm, Sweden, e-mail: ather.gattami@ericsson.com.}
\thanks{Manuscript received XX XX, 20XX; revised XX XX, 20XX.}}

\maketitle

\begin{abstract}
This paper proposes a direct, and simple approach to the $H_{\infty}$ norm calculation in more general settings.
In contrast to the method based on the Kalman--Yakubovich--Popov lemma, our approach does not require a controllability assumption, and returns a sinusoidal input that achieves the $H_{\infty}$ norm of the system including its frequency.
In addition, using a semidefinite programming duality, we present a new proof of the Kalman--Yakubovich--Popov lemma, and make a connection between strong duality and controllability.
Finally, we generalize our approach towards the generalized Kalman--Yakubovich--Popov lemma, which considers input signals within a finite spectrum.
\end{abstract}


\IEEEpeerreviewmaketitle

\input{./tex/robust_intro}
\input{./tex/discrete}
\input{./tex/bounded}
\input{./tex/conclusion}

\section*{ACKNOWLEDGMENTS}
The authors gratefully acknowledge the helpful discussions with Prof. John C. Doyle.
\bibliographystyle{IEEEtran}
\bibliography{IEEEabrv,gkyp}

\input{./tex/robust_appendix.tex}

\end{document}

%% file: tex/defs.tex
\newcommand{\BEAS}{\begin{eqnarray*}}
\newcommand{\EEAS}{\end{eqnarray*}}
\newcommand{\BEA}{\begin{eqnarray}}
\newcommand{\EEA}{\end{eqnarray}}
\newcommand{\BEQ}{\begin{equation}}
\newcommand{\EEQ}{\end{equation}}
\newcommand{\BIT}{\begin{itemize}}
\newcommand{\EIT}{\end{itemize}}
\newcommand{\BNUM}{\begin{enumerate}}
\newcommand{\ENUM}{\end{enumerate}}

\newcommand{\Tr}[1]{\mathop{\bf Tr}\left({#1}\right)}

\newcommand{\rank}[1]{\text{Rank}\left({#1}\right)}

%% file: tex/robust_intro.tex
\section*{Notation}
 $W^*$ is the Hermitian of $W$, $\Tr{W}$ is the trace of $W$, and $W^{\dagger}$ is the pseudoinverse of $W$. The generalized inequality, $X \succeq\, 0$, means $X$ is a positive semidefinite matrix. For positive semidefinite matrix $X$, $X^{1/2}$ is a matrix sqaure root such that $X = X^{1/2}\left(X^{1/2}\right)^*$. The matrix $I_m$ denotes the $m \times m$ identity matrix, and $0_{m,n}$ is the $m \times n$ zero matrix.
 
\section{Introduction}
We revisit the problem of computing the $H_{\infty}$ norm of a Linear Time Invariant (LTI) system.
This problem is well studied in the literature, and can be found in standard textbooks in control, \cite{boyd1994linear}, \cite{zhou1996robust}, and \cite{Dullerud:2010tc} to name a few.
Moreover, efficient methods for calculating $H_{\infty}$ norm are also investigated in \cite{boyd1990regularity}, \cite{bruinsma1990fast}, and \cite{scherer1990h}.

All these standard results are based on the Kalman--Yakubovich--Popov (KYP) lemma \cite{Rantzer:2011wn}.
Recently, \cite{gattamisimple} presented a new, simple approach to the $H_{\infty}$ analysis based on a covariance-like formulation, and 
\cite{syoulagrangian} uses Lagrangian duality to propose a new proof of the KYP lemma.

This paper combines these new insights from \cite{gattamisimple} and \cite{syoulagrangian} to reformulate the $H_{\infty}$ analysis problem.
Specifically, we present a semidefinite program that computes the $H_{\infty}$ norm of the LTI system, which is the same as the one in \cite{gattamisimple}.
However, we show its exactness based on \cite{syoulagrangian}.

This semidefinite program is known to the community as a "dual" of the KYP lemma \cite{balakrishnan2003semidefinite}, but in contrast to the KYP lemma, an optimal solution of our semidefinite program can generate a sinusoidal input that acheives the $H_{\infty}$ norm of the system, and we present an algorithm for this task.

In addition, we show that controllability implies strong duality.
Therefore, without controllability, the Lyapunov matrix $P$ which satisfies the Linear Matrix Inequality (LMI) from the KYP lemma may not exist, and this makes it hard to compute the $H_{\infty}$ norm using the KYP lemma. However, our semidefinite program is always exact regardless of controllability of the system.
In this sense, this approach is more direct and general compared to the KYP lemma.

Finally, we generalize our idea towards the generalized KYP lemma \cite{Iwasaki:kx}, where one can compute the $H_{\infty}$ norm for input signals restricted to a finite frequency range. 

%% file: tex/discrete.tex
\section{$H_{\infty}$ analysis for discrete time LTI systems}
\subsection{Problem formulation}
Consider the LTI system $\mathcal{M}$ given by
\BEAS
{x}_{k+1} &=& A{x}_{k}+B{w}_{k}\\
{z}_{k} &=& C{x}_{k} + D{w}_{k},
\EEAS
$A \in \mathbb{C}^{n \times n}$, $B \in \mathbb{C}^{n \times m}$, $C \in \mathbb{C}^{l \times n}$, and $D \in \mathbb{C}^{l \times m}$. 
Moreover we assume that $A$ is Schur stable, that is, the spectral radius of $A$ is less than unity.
Although the $H_{\infty}$ analysis considers the $\mathcal{L}_2$ gain of the system, it can be shown that this is equivalent to considering the power norm \cite{doyle1994mixed}.
Therefore for the ease of presentation, we proceed the analysis with the power norm
\BEAS
\|\mathbf{h}\|_p^2 = \lim_{N \rightarrow \infty} \frac{1}{N}\sum_{k=0}^{N-1} {h}_k^*{h}_k.
\EEAS
We define the $H_{\infty}$ norm of the system $\mathcal{M}$ as
\BEAS
\|\mathcal{M}\|_{\infty} = \underset{\|\mathbf{w}\|_p \leq 1}{\sup} \frac{\|\mathbf{z}\|_p}{\|\mathbf{w}\|_p}.
\EEAS
From the linearity of the system, the supremum is always achieved at $\|\mathbf{w}\|_p = 1$. Therefore, for LTI systems, the above definition is equivalent to
\BEAS
\|\mathcal{M}\|_{\infty} = \sup_{\|\mathbf{w}\|_p \leq 1} \|\mathbf{z}\|_p.
\EEAS

From the above definition, we have the following infinite dimensional optimization problem:

\begin{align}
\hspace{-20mm}\|\mathcal{M}\|_{\infty}^2 = ~ \underset{\mathbf{w},\mathbf{x}} {\text{max}}
~~ &\lim_{N \rightarrow \infty} \frac{1}{N}\sum_{k=0}^{N-1} {z}_k^*{z}_k\label{eq:dopt1}\\
 \text{s. t. }~~
 &{x}_{k+1} = A{x}_{k}+B{w}_{k}\label{eq:doptcs}\\
 &{z}_{k} = C{x}_{k} + D{w}_{k}\\
 &{x}_0 = 0\\
&\|\mathbf{w}\|_p \leq 1.\label{eq:doptce}
\end{align}
Notice that this problem as posed is intractable since we have countably infinite number of variables and constraints.

\subsection{Main Result}
Define the new variable $$V = \underset{n \rightarrow \infty}{\lim} \frac{1}{N}\sum_{k=0}^{N-1} \begin{bmatrix} {x}_{k} \\ {w}_k\end{bmatrix}\begin{bmatrix} {x}_{k} \\ {w}_k\end{bmatrix}^* \succeq 0.$$

Then, the objective function will only depend on this new matrix $V$:
\begin{equation}
\lim_{N \rightarrow \infty} \frac{1}{N}\sum_{k=0}^{N-1} {z}_k^*{z}_k = \Tr{\begin{bmatrix} C & D \end{bmatrix} V \begin{bmatrix} C & D \end{bmatrix}^*}.\label{eq:cov}
\end{equation}
From the dynamics ${x}_{k+1} = A{x}_{k}+B{w}_{k}$, we have $x_{k+1}x_{k+1}^* = (Ax_k+Bw_k)(Ax_k+Bw_k)^*$. By taking the infinite sum on both sides, we conclude that
\BEAS
\begin{bmatrix} I_n & 0_{n, m} \end{bmatrix} V \begin{bmatrix} I_n \\ 0_{m, n} \end{bmatrix} = \begin{bmatrix} A & B \end{bmatrix} V \begin{bmatrix} A^* \\ B^* \end{bmatrix},
\EEAS
and from $\|\mathbf{w}\|_p \leq 1$, we get $$\Tr{\begin{bmatrix} 0_{m, n} & I_m \end{bmatrix} V \begin{bmatrix} 0_{m, n}  & I_m \end{bmatrix}^*} \leq 1.$$
Notice that adding this redundant constraint to \eqref{eq:dopt1} does not change the optimal value and the solution of the problem:
\begin{align}
\underset{V,\mathbf{x},\mathbf{w}} {\text{max}}~~
& \Tr{\begin{bmatrix} C & D \end{bmatrix} V \begin{bmatrix} C & D \end{bmatrix}^*}\label{eq:dopt2}\\
 \text{s. t.}~~
 & \eqref{eq:doptcs} - \eqref{eq:doptce}\nonumber \\
 &\begin{bmatrix} I_n & 0_{n, m} \end{bmatrix} V \begin{bmatrix} I_n \\ 0_{m, n} \end{bmatrix} = \begin{bmatrix} A & B \end{bmatrix} V \begin{bmatrix} A^* \\ B^* \end{bmatrix}\label{eq:dopt2cs}\\
& \Tr{\begin{bmatrix} 0_{m, n} & I_m \end{bmatrix} V \begin{bmatrix} 0_{m, n}  & I_m \end{bmatrix}^*} \leq 1\\
& V \succeq 0\label{eq:dopt2ce}\\
& V = \underset{N \rightarrow \infty}{\lim} \frac{1}{N}\sum_{k=0}^{N-1} \begin{bmatrix} {x}_{k} \\ {w}_k\end{bmatrix}\begin{bmatrix} {x}_{k} \\ {w}_k\end{bmatrix}^*.
\end{align}
This is still problematic. First of all, \eqref{eq:dopt2} is still an infinite dimensional problem, and the last equality is not affine. Therefore this lifted problem \eqref{eq:dopt2} is an infinite-dimensional non-convex problem.
However, by dropping  \eqref{eq:doptcs} - \eqref{eq:doptce}, and the last equality, we have a relaxed version of $\eqref{eq:dopt2}$ which is a finite dimensional semidefinite program (SDP) \cite{boyd2004convex}:
\begin{align}
\mu_{\text{opt}}:=
\underset{V} {\text{max}}~~
& \Tr{\begin{bmatrix} C & D \end{bmatrix} V \begin{bmatrix} C & D \end{bmatrix}^*}\label{eq:dsdp}\\
\text{s. t.}~~ &\eqref{eq:dopt2cs} - \eqref{eq:dopt2ce}.\nonumber
\end{align}

One direct consequence of this relaxation is that the optimal value of \eqref{eq:dopt2} is less than $\mu_{\text{opt}}$, because \eqref{eq:dsdp} has a larger feasible set.
This gives us $\mu_{\text{opt}} \geq \|\mathcal{M}\|_{\infty}^2$.

However the following non-trivial result shows that $\mu_{\text{opt}}$ is actually the same as $\|\mathcal{M}\|_{\infty}^2$ and we can recover the optimal solution of \eqref{eq:dopt1} from the optimal solution $V_{\text{opt}}$ of \eqref{eq:dsdp}.

\begin{theorem}
The optimal value of \eqref{eq:dsdp}, $\mu_{\text{opt}}$, is equal to the optimal value of \eqref{eq:dopt1}, $\|\mathcal{M}\|_{\infty}^2$.
\label{thm:eq}
\end{theorem}
Notice that, unlike \cite{Rantzer:2011wn}, or \cite{gattamisimple}, \eqref{eq:dsdp} does not require the controllability to compute $H_{\infty}$ norm.
Moreover,  we will describe how to construct the optimal solution of \eqref{eq:dopt1}, which is the problem of interest, from the optimal solution of \eqref{eq:dsdp}, and this solution constructs the sinusoidal input $\mathbf{w}$ that achieves $\|\mathcal{M}\|_{\infty}^2$.
Before proving this result, we need some technical lemmas.

\begin{lemma}[Rank one decomposition]
Suppose that $V\succeq 0$, and $\begin{bmatrix} I_n & 0_{n,m} \end{bmatrix} V \begin{bmatrix} I_n\\ 0_{m,n} \end{bmatrix} = \begin{bmatrix} A& B \end{bmatrix} V \begin{bmatrix} A^*\\ B^* \end{bmatrix}$. Then there exists a set of rank one matrices $\{V_k\}$ such that
\BEAS
V = \sum_k V_k, ~\textup{Rank}(V_k) = 1, ~V_k \succeq 0,\\
\begin{bmatrix} I_n & 0_{n,m} \end{bmatrix} V_k \begin{bmatrix} I_n\\ 0_{m,n} \end{bmatrix} = \begin{bmatrix} A& B \end{bmatrix} V_k \begin{bmatrix} A^*\\ B^* \end{bmatrix}
\EEAS
\label{lemma:rank1-1}
\end{lemma}
\begin{proof}
The following proof is from \cite{Rantzer:2011wn}. 
Let $F = \begin{bmatrix} I_n & 0_{n,m} \end{bmatrix} V^{1/2}$, and $G = \begin{bmatrix} A & B \end{bmatrix} V^{1/2}$.
Since $FF^* = GG^*$, Lemma \ref{lemma:ra} in the appendix implies that there exists a unitary matrix $U$ such that $F = GU$. 
Being unitary, $U = \sum_k e^{j\theta_k}u_ku_k^*$, and $\sum_k u_ku_k^* = I$. Notice that $Fu_k = GUu_k = e^{j\theta_k}Gu_k$, and $Fu_ku_k^*F^* = Gu_ku_k^*G^*$. 
Therefore, by defining $V_k = V^{1/2}u_k\left(V^{1/2}u_k\right)^*$, it is routine to check that this construction satisfies all the constraints.
\end{proof}

The above lemma shows that the extreme points of the feasible set of \eqref{eq:dsdp} are rank one matrices.
Since the objective function in \eqref{eq:dsdp}, $\Tr{\begin{bmatrix} C & D \end{bmatrix} V \begin{bmatrix} C & D \end{bmatrix}^*}$, is {\it{linear}} in $V$, there exists a rank one optimal solution of \eqref{eq:dsdp}.
\begin{proposition}[Rank one optimal solution]
There exists an optimal solution $V_{\text{opt}}$ of \eqref{eq:dsdp} with $\textup{Rank}({V_{\text{opt}}}) = 1$.
\label{prop:rankone}
\end{proposition}
\begin{proof}
Let $V_{\text{opt}}$, and $\mu_{\text{opt}}$ be the optimal solution and the optimal value of \eqref{eq:dsdp}, respectively. From Lemma \ref{lemma:rank1-1}, there exist rank one matrices $V_k$ such that $V_{\text{opt}} = \sum_k V_k$ and $V_k$ is in the feasible set of \eqref{eq:dsdp}.

Define the scalar values:
\BEAS
p_k &=& \Tr{\begin{bmatrix} 0_{m,n} & I_m \end{bmatrix} V_k \begin{bmatrix} 0_{m,n} & I_m \end{bmatrix}^*},\\
\mu_k &=& \Tr{\begin{bmatrix} C & D \end{bmatrix} V_k \begin{bmatrix} C & D \end{bmatrix}^*}.
\EEAS
Then, we must have $\mu_{\text{opt}} = \sum_k \mu_k$ and $\sum_k p_k \leq 1$.
Let $J$ be the index such that
$$J = \underset{k}{\text{argmax }} \frac{\mu_k}{p_k}.$$

We will show that $\hat{V} = \frac{1}{p_J}V_J$ is a rank one optimal solution of \eqref{eq:dsdp}. 
It is easy to check that $\hat{V}$ is a feasible point of \eqref{eq:dsdp}.
Moreover,
\BEAS
&&\Tr{\begin{bmatrix} C & D \end{bmatrix} \hat{V} \begin{bmatrix} C & D \end{bmatrix}^*} = \frac{\mu_J}{p_J} \\
 &\geq& \sum_k p_k \frac{\mu_J}{p_J} \geq \sum_k p_k \frac{\mu_k}{p_k} = \sum_k \mu_k= \mu_{\text{opt}}.
\EEAS
Therefore, $\Tr{\begin{bmatrix} C & D \end{bmatrix} \hat{V} \begin{bmatrix} C & D \end{bmatrix}^*} \geq \mu_{\text{opt}}$ which implies that $\Tr{\begin{bmatrix} C & D \end{bmatrix} \hat{V} \begin{bmatrix} C & D \end{bmatrix}^*} = \mu_{\text{opt}}$, and $\hat{V}$ is a rank one optimal solution.
\end{proof}

Now we are ready to prove Theorem \ref{thm:eq}.

\begin{proof}[Proof of Theorem \ref{thm:eq}]
From Proposition \ref{prop:rankone}, we obtain a rank one optimal solution $V_{\text{opt}}$ of \eqref{eq:dsdp}.
Since Rank$(V_{\text{opt}}) = 1$, there exist a pair of vectors $x_{\text{opt}} \in \mathbb{C}^n$, and $w_{\text{opt}} \in \mathbb{C}^m$ such that $V_{\text{opt}} = \begin{bmatrix} x_{\text{opt}}\\w_{\text{opt}} \end{bmatrix}\begin{bmatrix} x_{\text{opt}}\\w_{\text{opt}} \end{bmatrix}^*$.
Since $V_{\text{opt}}$ satisfies \eqref{eq:dopt2cs}, we have that
\BEAS
&&\begin{bmatrix} I_n &  0_{n,m} \end{bmatrix} \begin{bmatrix} x_{\text{opt}}\\w_{\text{opt}} \end{bmatrix}\begin{bmatrix} x_{\text{opt}}\\w_{\text{opt}} \end{bmatrix}^* \begin{bmatrix} I_n\\ 0_{m,n} \end{bmatrix} \\
&=& \begin{bmatrix} A &  B \end{bmatrix} \begin{bmatrix} x_{\text{opt}}\\w_{\text{opt}}\end{bmatrix}\begin{bmatrix} x_{\text{opt}}\\w_{\text{opt}}\end{bmatrix}^* \begin{bmatrix} A^*\\ B^* \end{bmatrix},
\EEAS
and from Corollary \ref{cor:vector-1} in the appendix, there exists a scalar $\theta_{\text{opt}}$ such that
\BEAS
e^{j\theta_{\text{opt}}} x_{\text{opt}}= Ax_{\text{opt}} + Bw_{\text{opt}}.
\EEAS
Therefore, $w_k = e^{j\theta_{\text{opt}} k} w_{\text{opt}}$ results in ${x}_{k} = (e^{j\theta_{\text{opt}}k}I_n - A^k)x_{\text{opt}}$.
By substitution, we conclude that $\{w_k\}$ achieves $\mu_{\text{opt}}$. Therefore $\{w_k\}$ and $\{x_k\}$ are the optimal solutions of \eqref{eq:dopt1}.
\end{proof}

%
\begin{remark}
\normalfont
In the proof, we construct $\mathbf{w}$, the optimal solution of \eqref{eq:dopt1}, and this signal turns out to be a sinusoid. This is a well known fact in the literature, since the $H_{\infty}$ norm of the system is the maximum value in the Bode magnitude plot. 
However, in contrast to other approaches, {\it{e.g.}} \cite{Rantzer:2011wn}, we explicitly construct the input $\mathbf{w}$, and its spectrum $\theta_{\text{opt}}$.
\end{remark}

\subsection{The optimal input extraction}
As Remark 1 points out, we can construct the optimal input $\mathbf{w}$ by solving \eqref{eq:dsdp} based on our new proof.
Suppose we obtain a solution $V_{\text{opt}}$ of \eqref{eq:dsdp}.
If $V_{\text{opt}}$ is rank one, then it requires no additional step.
Simply find a pair of vectors $(x_{\text{opt}},w_{\text{opt}})$ such that $V_{\text{opt}} = \begin{bmatrix} x_{\text{opt}}\\w_{\text{opt}} \end{bmatrix}\begin{bmatrix} x_{\text{opt}}\\w_{\text{opt}} \end{bmatrix}^*$. Then, $\theta_{\text{opt}}$ such that $e^{j\theta_{\text{opt}}} x_{\text{opt}} = Ax_{\text{opt}}+Bw_{\text{opt}}$  is guaranteed to exist.
Therefore, we can use element-wise division between $x_{\text{opt}}$ and $Ax_{\text{opt}}+Bw_{\text{opt}}$ to find the spectrum $\theta_{\text{opt}}$.
Then, $w_k = e^{j\theta_{\text{opt}} k} w_{\text{opt}}$.

If $V_{\text{opt}}$ is not rank one, we can use the procedure in the proof of Proposition \ref{prop:rankone} to recover the rank one solution, then apply the aforementioned procedure.
To do this, we need a unitary matrix $U$ which satisfies
$$ \begin{bmatrix} I_n & 0_{n,m} \end{bmatrix} V_{\text{opt}}^{1/2} = \begin{bmatrix} A & B \end{bmatrix} V_{\text{opt}}^{1/2}U.$$
We modify the construction in \cite{iwasaki2000generalized}, to find such a unitary matrix $U$.
The correctness of this algorithm can be easily shown by substitution.

\noindent
\makebox[0.5\textwidth]{\rule{0.5\textwidth}{.1pt}}\\
\textbf{Algorithm 1}\\
\textbf{Input:} Complex matrices, $F,G$ such that $FF^* = GG^*$\\
\textbf{Output:} A unitary matrix $U$ such that $F = GU$\\
\makebox[0.5\textwidth]{\rule{0.5\textwidth}{.1pt}}\\
\begin{enumerate}
\item Set $P = F+G$, and $Q = F-G$
\item Find the SVD of $P = U_P\Sigma_PV_P^*$, and let $r = \textup{Rank}(P)$
\item Set $\begin{bmatrix}R&S\end{bmatrix} = \begin{bmatrix} I_r & 0 \end{bmatrix} V_P^*P^{\dagger}QV_P$
\item Set $\Delta = V_P\begin{bmatrix}R&S\\-S^* & 0\end{bmatrix}V_P^*$
\item $U = (I+\Delta)(I-\Delta)^{-1}$.
\end{enumerate}
\makebox[0.5\textwidth]{\rule{0.5\textwidth}{.1pt}}\\
\noindent
%
%
%
By applying the Algorithm 1 to $F = \begin{bmatrix} I_n & 0_{n,m} \end{bmatrix} V_{\text{opt}}^{1/2}$, and $G = \begin{bmatrix} A & B \end{bmatrix} V_{\text{opt}}^{1/2}$ we obtain a desired unitary matrix $U$ in the proof of Lemma \ref{lemma:rank1-1}. 
The second step is to perform eigenvalue decomposition of $U$ to have $U = \sum_k e^{j\theta_k} u_ku_k^*$, where $u_k$ is the eigenvector of $U$.
The third step is to find $V_k = V_{\text{opt}}^{1/2}u_k\left(V_{\text{opt}}^{1/2}u_k\right)^*$.
The final step is to find a index $J$ which maximizes $\frac{\mu_k}{p_k}$ as in the Proposition \ref{prop:rankone}. 
Then $V_J$ is a rank one optimal solution.

%
%
%

\subsection{Connection to the KYP lemma}

The Lagrangian dual of our optimization \eqref{eq:dsdp} generates the optimization derived from the KYP lemma:
\begin{equation}
\begin{aligned}
& \underset{\lambda, P} {\text{min}}
& & \lambda\\
& \text{s. t. }
& & \begin{bmatrix} A^*PA - P & A^*PB\\B^*PA & B^*PB \end{bmatrix}\\
&&&+ \begin{bmatrix} C^*C & C^*D\\D^*C & D^*D-\lambda I_m  \end{bmatrix}\preceq 0\\
&&& \lambda \geq 0, P = P^*.
\end{aligned}
\label{eq:dual-dsdp}
\end{equation}

However, there is no guarantee for the optimal value of \eqref{eq:dual-dsdp} to be the same as the $H_{\infty}$ norm of the system. 
In addition, even if the duality gap is zero, there may be no dual feasible point $(P, \lambda)$ achieving the dual optimum value.
This is the reason why the KYP lemma with a non-strict inequality requires an additional assumption \cite{Rantzer:2011wn}, controllability of $(A,B)$, to guarantee the existence of a Lyapunov matrix $P$.

The following example shows the case when strong duality fails.
\begin{example}
Consider the scalar system with $(A,B,C,D) = (0,0,1,1)$. The optimal value of \eqref{eq:dsdp} is $1$, and so as $\|\mathcal{M}\|_{\infty} = 1$,
and the dual optimal value of \eqref{eq:dual-dsdp} is also $1$. However, the dual optimal solution does not exist because $(P, \lambda) \rightarrow (+\infty, 1)$ generates the dual optimal value.
\end{example}
%
%
%

The following proposition claims that the controllability is a sufficient condition for strong duality.
\begin{proposition}
Suppose $(A,B)$ is controllable. Then, strong duality holds between \eqref{eq:dsdp} and \eqref{eq:dual-dsdp}.
\label{prop:strongdual}
\end{proposition}
The basic idea is to construct a positive definite feasible point $V \succ 0$ for \eqref{eq:dsdp}, and this shows that Slater's condition for strong duality is satisfied.
Since the construction is very technical, we relegate the proof to the appendix.

%% file: tex/bounded.tex
\section{Bounded Frequency $H_{\infty}$ Analysis}
As we have seen, the spectrum of the input that achieves the $H_{\infty}$ norm lies in $[-\pi,\pi]$.
In this section, we consider inputs with bounded frequency. 
Specifically, for a given frequency $0 < \theta_0 < \pi$, we consider the input with the specific form: $\mathcal{W}_L = \{\mathbf{w}:w_k = e^{j\theta k} w_s$, $\theta \in [-\theta_0,\theta_0]\}$.
This formulation finds the maximum value of the Bode magnitude plot of the system in the low frequency region, $[-\theta_0,\theta_0]$, not in the entire region, $[-\pi,\pi]$:
%
\begin{align}
\underset{\mathbf{w},\mathbf{x}} {\text{max}}
~~ &\lim_{N \rightarrow \infty} \frac{1}{N}\sum_{k=0}^{N-1} {z}_k^*{z}_k\label{eq:dopt-low}\\
 \text{s. t. }~~ &\eqref{eq:doptcs} - \eqref{eq:doptce} \nonumber\\
 & \mathbf{w} \in \mathcal{W}_L.
\end{align}


Consider the input $w_k = e^{j\theta k}w_0 \in \mathcal{W}_L$.
This results in $x_k = e^{j\theta k}s + h_k$, where 
$$e^{j\theta} s = As + Bw_0,$$
and $h_k = - A^ks$ is a transient term which goes to zero asymptotically.
Notice that from the dynamics, we have that
\BEAS
x_{k+1}x_k^* + x_kx_{k+1}^* = (Ax_k+Bw_k)x_k^* + x_k(Ax_k+Bw_k)^*,
\EEAS
and by substituting $x_k = e^{j\theta k}s + h_k$, we get
\BEAS
x_{k+1}x_k^* + x_kx_{k+1}^* &=& 2\cos\theta ss^*  + e^{j\theta k}(1+e^{j\theta}) sh_k^* \\
&&+ e^{-j\theta k}(1+e^{-j\theta}) h_ks^*+h_kh_k^*.
\EEAS
This implies that $V$ in \eqref{eq:cov} satisfies 

\BEAS
&&2\cos\theta \begin{bmatrix} I_n & 0_{n,m} \end{bmatrix}V \begin{bmatrix} I_n \\ 0_{m,n} \end{bmatrix} \\
&=& \begin{bmatrix} A & B \end{bmatrix} V \begin{bmatrix} I_n \\ 0_{m,n} \end{bmatrix}+ \begin{bmatrix} I_n & 0_{n,m} \end{bmatrix} V \begin{bmatrix}A^*\\B^*\end{bmatrix}.
\EEAS
In addition, since $\theta \in [-\theta_0,\theta_0]$ and $\theta_0 \leq \pi$, we have $\cos\theta \geq \cos \theta_0$, and 
\begin{eqnarray}
&&2\cos\theta_0 \begin{bmatrix} I_n & 0_{n,m} \end{bmatrix}V \begin{bmatrix} I_n \\ 0_{m,n} \end{bmatrix}\nonumber \\
&\preceq& \begin{bmatrix} A & B \end{bmatrix} V \begin{bmatrix} I_n \\ 0_{m,n} \end{bmatrix}+ \begin{bmatrix} I_n & 0_{n,m} \end{bmatrix} V \begin{bmatrix}A^*\\B^*\end{bmatrix}. \label{eq:lowcov}
\end{eqnarray}

Therefore, a relaxation of \eqref{eq:dopt-low} gives
\begin{align}
\mu^{L}_{\text{opt}}:=
\underset{V} {\text{max}}~~
& \Tr{\begin{bmatrix} C & D \end{bmatrix} V \begin{bmatrix} C & D \end{bmatrix}^*}\label{eq:dsdp-low}\\
\text{s. t.}~~ &\eqref{eq:dopt2cs} - \eqref{eq:dopt2ce}, \eqref{eq:lowcov}\nonumber
\end{align}
Notice that we add \eqref{eq:lowcov} to \eqref{eq:dsdp} because of the constraint $\mathbf{w} \in \mathcal{W}_L$.

Since this is a relaxed version of \eqref{eq:dopt-low}, the optimal value of \eqref{eq:dsdp-low} yields an upper bound of \eqref{eq:dopt-low}, and in fact, this upper bound is tight.
 
\begin{proposition}
The optimal value of \eqref{eq:dsdp-low} equals the optimal value of \eqref{eq:dopt-low}.
\end{proposition}
The proof is almost identical once we have a rank one decomposition of $V_{\text{opt}}$.
For notational simplicity, let $\mathcal{F}_L$ be the feasible set of \eqref{eq:dsdp-low}.
\begin{lemma}[Rank one decomposition]
For all $V \in \mathcal{F}_L$, there exists a set of rank one matrices, $V_k \in \mathcal{F}_L$ such that,
\BEAS
V = \sum_k V_k, ~\textup{Rank}(V_k) = 1.
\EEAS
\label{lemma:rank1-low}
\end{lemma}
\begin{proof}
Define $F = \begin{bmatrix} I_n & 0_{n,m} \end{bmatrix} V^{1/2}$, and $G = \begin{bmatrix} A & B \end{bmatrix} V^{1/2}$.
Then from Lemma \ref{lemma:syou} in the appendix, there exists a unitary matrix $U$ such that $F = GU$, and $U+U^* \succeq 2\cos\theta I$.

Being unitary, $U = \sum_{k}e^{j\theta_k}u_ku_k^*$. By defining $V_k = V^{1/2}u_k\left(V^{1/2}u_k\right)^*$, we can easily check that $V_k \in \mathcal{F}_L$, and $V = \sum_k V_k$.
\end{proof}

To extract $w_k$ from $V_{\text{opt}}$, we can use exactly same procedure as in Section II. except finding a unitary matrix $U$ due to additional requirement $U+U^* \succeq 2 \cos\theta I$.
We modify an algorithm from \cite{iwasaki2000generalized} to find a desired $U$.

\noindent
\makebox[0.5\textwidth]{\rule{0.5\textwidth}{.1pt}}\\
\textbf{Algorithm 2}\\
\textbf{input:} Complex matrices, $F,G$ such that $FF^* = GG^*$, $FG^*+GF^* \succeq 2\cos\theta I$.\\
\textbf{Output:} A unitary matrix $U$ such that $F = GU$, and $U+U^* \succeq 2 \cos\theta I$.\\
\makebox[0.5\textwidth]{\rule{0.5\textwidth}{.1pt}}\\
\begin{enumerate}
\item Set $\mu = \frac{1-\cos\theta}{1+\cos\theta}$
\item Set $P = \sqrt{\mu}(F+G)$, and $Q = F-G$
\item Find the SVD of $P = U_P\Sigma_PV_P^*$, and let $r = \textup{Rank}(P)$
\item Set $\begin{bmatrix}R&S\end{bmatrix} = \begin{bmatrix} I_r & 0 \end{bmatrix} V_P^*P^{\dagger}QV_P$
\item Set $\Delta = V_P\begin{bmatrix}R&S\\-S^* & -S^*R(I_r+R^2)^{\dagger}S\end{bmatrix}V_P^*$
\item $U = (I+\sqrt{\mu}\Delta)(I-\sqrt{\mu}\Delta)^{-1}$
\end{enumerate}
\makebox[0.5\textwidth]{\rule{0.5\textwidth}{.1pt}}\\
\noindent

Also, notice that from Lemma \ref{lemma:syou}, $\theta_{\text{opt}}$ is guaranteed to be in $[-\theta_0,\theta_0]$.

%
%
%
\subsection{Connection to the Generalized KYP lemma}
Following is the Lagrangian dual problem of \eqref{eq:dsdp-low}:
\begin{equation}
\begin{aligned}
& \underset{\lambda, P} {\text{min}}
& & \lambda\\
& \text{s. t. }
&& \begin{bmatrix} A&B\\I&0\end{bmatrix}^*\begin{bmatrix} P&Q\\Q&-P-2\cos\theta_0Q\end{bmatrix}\begin{bmatrix} A&B\\I&0\end{bmatrix}\\
&&& + \begin{bmatrix} C^*C & C^*D\\D^*C & D^*D\end{bmatrix} \preceq \lambda \begin{bmatrix} 0& 0\\0 & I \end{bmatrix}\\
&&& \lambda \geq 0, P = P^*, Q \succeq 0,
\end{aligned}
\label{eq:dual-dsdp-low}
\end{equation}
and this can be also derived from the Generalized KYP lemma \cite{Iwasaki:kx}.
However, as in the $H_{\infty}$ analysis, the strong duality issue arises.

One condition for strong duality is controllability of $(A,B)$. Since the proof is identical with Proposition \ref{prop:strongdual}, we omit the details.

\subsection{$H_{\infty}$ Analysis with high frequency input}
Using similar arguments as in the low frequency input case, we can also include the cases with high frequency inputs, $\mathcal{W}_H = \{\mathbf{w}:w_k = e^{j\theta k} w_s, \theta \in [-\pi,-\theta_0] \cup [\theta_0,\pi]\}$, or the middle frequency inputs, $\mathcal{W}_M = \{\mathbf{w}:w_k = e^{j\theta k} w_s, \theta \in [\theta_1,\theta_2]\}$.
%
For high frequency inputs, $\mathcal{W}_H$, a similar approach as in the low frequency input gives us
\begin{align}
\mu^{{H}}_{\text{opt}}:=
\underset{V} {\text{max}}~~
& \Tr{\begin{bmatrix} C & D \end{bmatrix} V \begin{bmatrix} C & D \end{bmatrix}^*}\label{eq:dsdp-high}\\
\text{s. t.}~~ &\eqref{eq:dopt2cs} - \eqref{eq:dopt2ce}\nonumber\\
& \begin{bmatrix} A & B \end{bmatrix} V \begin{bmatrix} I_n \\ 0_{m,n} \end{bmatrix}+ \begin{bmatrix} I_n & 0_{n,m} \end{bmatrix} V \begin{bmatrix}A^*\\B^*\end{bmatrix}\nonumber\\
&\preceq 2\cos\theta_0 \begin{bmatrix} I_n & 0_{n,m} \end{bmatrix} V \begin{bmatrix} I_n \\ 0_{m,n} \end{bmatrix} \label{eq:covhigh}.
\end{align}

For middle frequency inputs, $[\theta_1,\theta_2]$, we can use \eqref{eq:dsdp-low} by shifting $B,D$.
Define $\theta_c = \frac{1}{2}(\theta_1+\theta_2)$, and $\theta_0 =  \frac{1}{2}(\theta_2-\theta_1)$.
Then the spectrum of $\mathbf{w}$ confined in $[\theta_1,\theta_2]$ is equivalent to that the new input $\tilde{\mathbf{w}} = e^{-j\theta_c}\mathbf{w}$ has finite spectrum on $[-\theta_0,\theta_0]$. In this coordinate, $B\mathbf{w} = Be^{j\theta_c} \tilde{\mathbf{w}}$, and $D\mathbf{w} = De^{j\theta_c} \tilde{\mathbf{w}}$.
Therefore by defining $\tilde{B} = Be^{j\theta_c}$, and $\tilde{D} = De^{j\theta_c}$ then using \eqref{eq:dsdp-low}, we can use an SDP to compute $H_{\infty}$ norm over the middle frequency range.

As a final remark, we could also derive dual problems that are equivalent to the Generalized KYP lemma.
However, since the derivation is similar, we omit the details.

%% file: tex/conclusion.tex
\section{Conclusion}
In this paper, we proposed a simple, direct approach to $H_{\infty}$ analysis. 
Compared to the classical approach based on the  KYP lemma, our approach can construct an explicit input signal that achieves the $H_{\infty}$ norm of the system, and does not require the controllability condition of the system to calculate the $H_{\infty}$ norm.
Moreover, we generalize this approach to the low, middle and high frequency input signals and show the effectiveness of our new approach.




%% file: tex/robust_appendix.tex
\section*{Appendix}
\subsection{Results from Linear Algebra}
\begin{lemma}[A. Rantzer, 1996]
Let $F,G \in \mathbb C^{n \times m}$.
The following statements are equivalent.\\
(i) $FF^* = GG^*$.\\
(ii) There exists a unitary matrix $U$ such that, $F=GU$.
\label{lemma:ra}
\end{lemma}
\begin{proof}
See \cite{Rantzer:2011wn}.
\end{proof}

For a special case of Lemma \ref{lemma:ra}, consider $f, g \in\mathbb{C}^{n \times 1}$. In this case, a unitary matrix $U$ is actually scalar, and we get following immediate consequence.
\begin{corollary}
$ff^* = gg^*$ if and only if $f = e^{j\theta} g$ for some $\theta$.
\label{cor:vector-1}
\end{corollary}

\begin{lemma}[T. Iwasaki, 2000]
The following statements are equivalent.\\
(i) $FF^* \preceq GG^*$ and $FG^*+GF^* = 0$.\\
(ii) There exists a skew-symmetric matrix $\Delta = -\Delta^*$ such that,
$F = G\Delta, \|\Delta\| \leq 1$.
\label{lemma:iwasaki}
\end{lemma}
\begin{proof}
See \cite{Ebihara:2009uc} or \cite{iwasaki2000generalized}
\end{proof}

The next result is a consequence of Lemma \ref{lemma:iwasaki}.
\begin{lemma}
The following statements are equivalent.\\
(i) $FF^* = GG^*$, and $FG^*+GF^* \succeq 2\cos\theta FF^*$.\\
(ii) There exists a unitary matrix $U$ such that,
$F = GU$ and, $U+U^* \succeq 2\cos\theta I$.
\label{lemma:syou}
\end{lemma}
\begin{proof}
From the direction (ii) to (i) is trivial.
Let us show the direction from (i) to (ii).

Define $\mu = \frac{1-\cos\theta}{1+\cos\theta} < 1$, $P = F-G$, and $Q=\sqrt{\mu}(F+G)$. Then, (i) is equivalent to
\BEAS
PQ^*+QP^* = 0, \qquad PP^* \preceq QQ^*.
\EEAS
From Lemma \ref{lemma:iwasaki}, there exists a matrix $\Delta$ such that
\BEAS
P = Q\Delta, \qquad \|\Delta\| \leq 1, \qquad \Delta + \Delta^* = 0.
\EEAS
Since $\Delta$ is skew-symmetric, we can find a unitary matrix $S$ such that
\BEAS
\Delta = S~\text{diag}{\{j\lambda_i\}}~S^*,
\EEAS
where $j\lambda_i$ is the $i$th eigenvalue of $\Delta$. From the condition $\|\Delta\| \leq 1$, we have $|\lambda_i|\leq 1$.

Now let us define $U = S~\text{diag}\{\frac{1+j\sqrt{\mu}\lambda_i}{1-j\sqrt{\mu}\lambda_i}\}~S^*$.
Notice that $U^* = S~\text{diag}\{\frac{1-j\sqrt{\mu}\lambda_i}{1+j\sqrt{\mu}\lambda_i}\}~S^*$.
Then, it's obvious that $UU^* = I$.
Notice that $$\frac{1+j\sqrt{\mu}\lambda_i}{1-j\sqrt{\mu}\lambda_i} + \frac{1-j\sqrt{\mu}\lambda_i}{1+j\sqrt{\mu}\lambda_i} = 2\frac{1-\mu\lambda_i^2}{1+\mu\lambda_i^2},$$ and since $\lambda_i^2 \leq 1$, we have $$\frac{1-\mu\lambda_i^2}{1+\mu\lambda_i^2} \geq \frac{1-\mu\lambda_i^2}{1+\mu\lambda_i^2}.$$
Therefore,
\BEAS
U+U^* &=& S~\text{diag} \{2\frac{1-\mu\lambda_i^2}{1+\mu\lambda_i^2}\}~S^*\\
	&\succeq& 2\frac{1-\mu}{1+\mu}SS^* = 2\cos\theta I,
\EEAS
which concludes the proof.
\end{proof}
Again, for a special case of Lemma \ref{lemma:syou}, consider $f, g \in\mathbb{C}^{n \times 1}$. We get following immediate consequence.
\begin{corollary}
$ff^* = gg^*$, $fg^*+gf^* \preceq 2\cos\theta_0 ff^*$, if and only if $f = e^{j\theta} g$ for some $\theta \in [-\theta_0,\theta_0]$.
\label{cor:vector-2}
\end{corollary}
\subsection{Removing Singularity}
\begin{proposition}
Let $A$ be Schur stable. Consider $V \succeq 0$ such that
$$\begin{bmatrix} I & 0 \end{bmatrix} V \begin{bmatrix} I \\ 0 \end{bmatrix} = \begin{bmatrix} A & B \end{bmatrix} V \begin{bmatrix} A^* \\ B^* \end{bmatrix}, \text{ and }~ \textup{Rank}(V) = 1,$$ 
then, $\Tr{\begin{bmatrix} 0 & I \end{bmatrix} V \begin{bmatrix} 0 & I \end{bmatrix}^*} > 0$.
\label{prop:non-singular}
\end{proposition}
\begin{proof}
From the decomposition, $V = \begin{bmatrix}x\\w\end{bmatrix}\begin{bmatrix}x\\w\end{bmatrix}^*$, if $\Tr{\begin{bmatrix} 0 & I \end{bmatrix} V \begin{bmatrix} 0 & I \end{bmatrix}^*} = 0$, then $w = 0$.
Moreover from Corollary \ref{cor:vector-1}, there exists $e^{j\theta}$ such that $e^{j\theta}x = Ax + Bw = Ax$, since $w = 0$. 
However, since $A$ is Schur stable, $e^{j\theta}I -A$ is invertible, therefore $x = 0$. This implies $V = 0$ which is contradict to $\text{Rank}(V) = 1$.
\end{proof}

\subsection{Proof of Proposition \ref{prop:strongdual}}
\begin{proof}
We find a basis $\{{v}_i\}$ for $\mathbb{C}^{n+m}$ to construct
\BEAS
V = \sum_{i=1}^{n+m} {v}_i {v}_i^*,
\EEAS
which satisfies the strict inequality constraints in (\ref{eq:dsdp}). Then the Slater's constraint qualification gives strong duality.

Firstly, pick $n+1$ numbers on the unit disk $e^{j\theta_0}, e^{j\theta_1}, \cdots, e^{j\theta_n}$, that are distinct\footnote{$\theta_i$ can be chosen in the specific frequency region to generalize the proof to bounded frequency $H_{\infty}$ analysis case.}.

Since $(A,B)$ is controllable, there exists a matrix $K$ such that the eigenvalues of $A-BK$ are $e^{j\theta_1}, \cdots, e^{j\theta_n}$.
Denote corresponding eigenvector ${x}_i$,
\begin{equation}
(A-BK){x}_i = e^{j\theta_i} {x}_i.
\label{feedback}
\end{equation}
for $i = 1,\cdots,n$. Moreover, let $T = \left[\begin{array}{c c} A-e^{j\theta_0}I & B \end{array}\right]$ is $n$. Then $\rank{T} = n$, because of the Popov-Belevitch-Hautus (PBH) controllability test \cite{Dullerud:2010tc}.
Therefore, there exists a basis $\{{t}_1, \cdots, {t}_m\}$ for $N(T)$, the null space of $T$, and by substitution, we can show that ${t}_i {t}_i^*$ is a feasible point of problem (\ref{eq:dsdp}), for all $i = 1, \cdots, m$.

Define ${v}_{n+k} = {t}_k$, and let $S_1 = $ span$({v}_1,,{v}_n)$, and $S_2 = N(T)$. 
Suppose,
\BEAS
{v} &=& \left[\begin{array}{c} {x} \\ {u} \end{array}\right] \in S_1 \cap S_2.
\EEAS
Since ${v} \in S_1$, there exists $\{\alpha_i\}_{i=1}^n$ such that ${v} = \sum_{i=1}^n \alpha_i {v}_i$,
which implies ${x} = \sum_{i=1}^n \alpha_i{x}_i$, and ${u} = \sum_{i=1}^n \alpha_i{w}_i$.
Furthermore, $e^{j\theta_0}{x} = A{x} + B{u}$, because ${v} \in S_2$. Combining these two equations, we have,
\BEAS
e^{j\theta_0}{x} &=& e^{j\theta_0}  \sum_{i=1}^n \alpha_i {x}_i,\\
A{x} + B{u} &=&  A\sum_{i=1}^n \alpha_i{x}_i + B \sum_{i=1}^n \alpha_i {u}_i\\
					&=& \sum_{i=1}^n \alpha_i e^{j\theta_i}{x}_i,
\EEAS
which implies,
\BEAS
\sum_{i=1}^n \alpha_i e^{j\theta_0} {x}_i &=& \sum_{i=1}^n \alpha_i e^{j\theta_i} {x}_i.
\EEAS
From this, we can conclude that $\alpha_i = 0$, for all $i$, because $\{{x}_i\}_{i=1}^n$ are linearly independent, and $\theta_i \neq \theta_0$ for all $i$.
Therefore, $S_1 \cap S_2 = \{0\}$ which implies $\{{v}_i\}_{i=1}^{n+m}$ are linearly independent.

Now it is easy to check that $V$ is a feasible point of \eqref{eq:dsdp} and that $V \succ 0$. This implies that Slater's constraint qualification holds \cite{boyd2004convex}, and we are done.
\end{proof}